\documentclass[12pt,a4paper]{article}

\usepackage{cite}
\usepackage{amssymb}
\usepackage{amsmath}
\usepackage{amscd}
\usepackage{graphicx}
\usepackage{amsthm}
\usepackage{amsfonts}

\newtheorem{Th}{Theorem}%[section]
\newtheorem{Lm}{Lemma}%[section]
\newtheorem{Exam}{Example}%[section]
\newtheorem{Rem}{Remark}%[section]
\newenvironment{Ex}{\begin{Exam}\normalfont}{\end{Exam}}

\title
{{Landesman-Lazer condition revisited: the influence \\
of vanishing and oscillating nonlinearities}}

\author
{ Pavel Dr\'abek \thanks{Department of Mathematics and NTIS, University of West Bohemia, Univerzitn\'i 8, 306 14 Plze\v{n}, Czech Republic, email:
pdrabek@kma.zcu.cz; corresponding author} ,
Martina
Langerov\'a \thanks{NTIS, University of West Bohemia, Univerzitn\'i 8, 306 14 Plze\v{n}, Czech Republic, email: mlanger@ntis.zcu.cz}}

\date{}

\begin{document}

\maketitle

\vspace{3mm}

\noindent{\fontsize{11pt}{0.17in} \emph {\textbf{Abstract.}
In this paper we deal with semilinear problems at resonance. We present a sufficient condition for the existence of a weak solution in terms of the asymptotic properties of nonlinearity. Our condition generalizes the classical Landesman-Lazer condition but it also covers the cases of vanishing and oscillating nonlinearities.}
\vspace{3mm}
}

\vspace{1cm}

%%%%%%%%%%%%%%%%%%%%%%%%%%%%%%%%%%%%%%%%%%%%%%%%%%%%%%%%%%%%%%%%%%%%%%%%%%%%%%%%%%%%%%%%%%%%%%%%%%%%%%%%%%%%%%%%%%
%%%%%%%%%%%%%%%%%%%%%%%%%%%%%%%%%%%%%%%%%%%%%%%%%%%%%%%%%%%%%%%%%%%%%%%%%%%%%%%%%%%%%%%%%%%%%%%%%%%%%%%%%%%%%%%%%%% Section 1
\section{Introduction}

Let $\Omega \subseteq \mathbb{R}^{n}$ be a bounded domain, $g: \mathbb{R} \rightarrow \mathbb{R}$ be a bounded continuous function and $f \in L^{2}(\Omega)$. We consider the boundary value problem
\begin{equation}\label{11}
\begin{aligned}
- \Delta u - \lambda_{k} u + g(u)& = f \quad  {\text {in}} \; \Omega ,  \\
u & = 0 \quad  {\rm on} \; \partial\Omega .
\end{aligned}
\end{equation}
Here $\lambda_{k}$, $k\geq 1$, is the $k$-th eigenvalue of the eigenvalue problem
\begin{equation}\label{12}
\begin{aligned}
- \Delta u - \lambda u & = 0 \quad  {\text {in}} \; \Omega ,  \\
u & = 0 \quad  {\rm on} \; \partial\Omega .
\end{aligned}
\end{equation}
By a \textit{solution} of \eqref{11} we understand a function
$u \in H:= W^{1,2}_{0}(\Omega)$ satisfying \eqref{11} in the weak sense, \textit{i.e.},
\begin{equation}\label{13}
\int\limits_{\Omega} \nabla u \nabla v \, {\rm d}x- \lambda_{k}
\int\limits_{\Omega} u v \, {\rm d} x + \int\limits_{\Omega} g(u) v \, {\rm d} x =
\int\limits_{\Omega} f v \, {\rm d} x
\end{equation}
holds for any test function $v \in H$.

Let $ m \geq 1$ be a multiplicity of $\lambda_{k}$. We arrange the eigenvalues of \eqref{12} into the increasing sequence:
$$ 0< \lambda_{1} < \lambda_{2} \leq \ldots \leq \lambda_{k-1} < \lambda_{k} = \ldots = \lambda_{k+m-1} < \lambda_{k+m} \leq \lambda_{k+m+1} \leq \ldots \rightarrow \infty.$$
The corresponding eigenfunctions, $(\phi_{n})$, form an orthogonal basis for both $L^{2}(\Omega)$ and $H$. We assume
that every $\phi_{n}$ is normalized with respect to the $L^{2}$ norm, \textit{i.e.}, $\| \phi_{n} \|_{2} =1$, $n = 1, 2, \ldots $ .

We use the scalar product $(u,v) = \int\limits_{\Omega} \nabla u \nabla v \, {\rm d}x$ and
the induced norm $\| u \| = \left (\int\limits_{\Omega} |\nabla u|^{2} \, {\rm d}x \right )^{\frac{1}{2}}$ on $H$. We split
the space $H$ into the following three subspaces spanned by the eigenfunctions of \eqref{12} as follows:
$$\hat{H} := [\phi_{1}, \ldots , \phi_{k-1}], \quad \bar{H} := [\phi_{k}, \ldots , \phi_{k+m-1}],\quad \tilde{H} := [ \phi_{k+m}, \phi_{k+m+1}, \ldots ]. $$
Then $H = \hat{H} \oplus \bar{H} \oplus \tilde{H}$ with $\dim \hat{H} = k-1$, $\dim \bar{H} = m$, $\dim \tilde{H} = \infty$.
Of course, if $k =1$ then $m =1$  and $\hat{H} = \emptyset$. We split an element $u \in H$ as $u = \hat{u} + \bar{u} + \tilde{u}$, $\hat{u} \in \hat{H}, \bar{u} \in \bar{H}$ and $\tilde{u} \in \tilde{H}$. A function $f \in L^{2}(\Omega)$ we split as
$f = \bar{f} + f^{\bot}$, where $\int\limits_{\Omega} f^{\bot} v \, {\rm d} x = 0$ for any $v \in \bar{H}$. The purpose of this paper is to introduce rather general sufficient condition of Landesman-Lazer type for the existence of a solution of \eqref{11}:

\vspace{3mm}
\textit{If $(u_{m}) \subset H$ is a sequence such that $\| u_{n} \|_{2} \rightarrow \infty$ and
there exists $\phi_{0} \in \bar{H}$, $\frac{u_{n}}{\| u_{n} \|_{2}} \rightarrow \phi_{0}$ in $L^{2}(\Omega)$, then} 
\begin{equation}\label{SC}\tag*{$(\rm SC)_{\pm}$}
\lim\limits_{n \rightarrow \infty} \left( \int\limits_{\Omega} G(u_{n}) \, {\rm d} x - \int\limits_{\Omega} \bar{f} u_{n} \, {\rm d} x \right) = \pm \infty.
\end{equation}
Here, $G(s) = \int\limits_{0}^{s} g(\tau)\, {\rm d} \tau$ is the antiderivative of $g$.
\begin{Th} \label{T11}
Assume that either $(\rm SC)_{+}$ or else $(\rm SC)_{-}$ holds. Then the problem \eqref{11} has at least one solution.
\end{Th}

\begin{Rem}
{ \rm Note that the sufficient condition which is similar to $(\rm SC)_{+}$ but more restrictive than $(\rm SC)_{+}$
was introduced recently in \cite{DR} where the resonance problem with respect to the Fu\v{c}\'{i}k
spectrum of the Laplacian
was studied. In this paper, we benefit from the fact that the resonance occurs at the eigenvalue
which allows us to split the underlying function space $H$ into the sum of orthogonal subspaces.
In contrast with \cite{DR}, where such splitting is impossible, we can get rid of the $f^{\bot}$-part of the right-hand
side $f$ in \ref{SC}. This makes our conditions more general and geometrically more transparent.}
\end{Rem}

In order to interpret our conditions \ref{SC} in historical context, we first consider a bounded continuous nonlinear
function $g: \mathbb{R} \rightarrow \mathbb{R}$ with finite limits $g(\pm \infty) := \lim\limits_{s \rightarrow \pm \infty} g(s)$.
Let us assume that
\begin{equation}\label{LL}\tag*{$(\rm LL)_{\pm}$}
\begin{aligned}
g(\mp \infty) \int\limits_{\Omega} \phi^{+} \, {\rm d} x - g(\pm \infty) \int\limits_{\Omega} \phi^{-} \, {\rm d} x
& < \int\limits_{\Omega} \bar{f} \phi \, {\rm d} x \\
& < g(\pm \infty) \int\limits_{\Omega} \phi^{+} \, {\rm d} x - g(\mp \infty) \int\limits_{\Omega} \phi^{-} \, {\rm d} x
\end{aligned}
\end{equation}
holds for all eigenfunctions $\phi$ associated with $\lambda_{k}$. This is the classical Landesman-Lazer condition (see \cite{LL}). Assume $\| u_{n}\|_{2} \rightarrow \infty$ and
$\frac{u_{n}}{\| u_{n} \|_{2}} \rightarrow \phi_{0}$ for some eigenfunction $\phi_{0}$. Then by l'Hospital's rule we have
\begin{equation*}%\label{LL}
\begin{aligned}
\lim\limits_{n \rightarrow \infty} \frac{1}{\| u_{n}\|_{2}} \left(\int\limits_{\Omega} G(u_{n}) \, {\rm d}x - \int\limits_{\Omega} \bar{f} u_{n} \, {\rm d} x \right)
 = \lim\limits_{n \rightarrow \infty} & \int\limits_{\Omega} \left( \frac{G(u_{n})}{u_{n}} - \bar{f}\right)
\frac{u_{n}}{\| u_{n}\|_{2}} \, {\rm d} x \\
 = \int\limits_{\Omega} \left( g(+ \infty) + \bar{f} \right) \phi^{+}_{0} & \, {\rm d} x
- \int\limits_{\Omega} \left( g(- \infty) + \bar{f} \right) \phi^{-}_{0} \, {\rm d} x.
\end{aligned}
\end{equation*}
The last expression is either positive or negative due to \ref{LL} and hence \ref{SC} hold.
In other words we proved that \ref{LL} imply \ref{SC}.

Assume, moreover, $g(- \infty) < 0 < g(+ \infty)$ (think, for example, about $g(s) = \arctan s$).
Then problem \eqref{11} has a solution for all $f$ which belong to the "strip" around the linear
subspace $L^{2}(\Omega)^{\bot} := \left \{ f \in L^{2}(\Omega): \int\limits_{\Omega} f \phi \, {\rm d} x = 0
\; {\rm for \; all} \; \phi \in \bar{H} \right \}$ of $L^{2}(\Omega)$.

We note that the conditions \ref{LL} are empty if $g(- \infty) = g(+ \infty)$. On the other hand, it follows
from Theorem \ref{T11} that the problem \eqref{11} with $g(s) = \frac{{\rm sgn} s}{(e+ |s|) \ln (e+ |s|)}$ ($e$ is Euler's number) has at least one
solution for $f \in L^{2}(\Omega)^{\bot}$. Indeed,
$\lim\limits_{|s| \rightarrow \infty} G(s) = \lim\limits_{|s| \rightarrow \infty} \ln \left( \ln (e+ |s|)\right) = \infty$
implies that $(\rm SC)_{+}$ holds true. Hence \ref{SC} cover the case of \textit{vanishing} nonlinearities
$g(\pm \infty) = 0$ (see \cite{D}). However, it should be emphasized, that in contrast with previous
works on vanishing nonlinearities our approach does not require any kind of symmetry or sign condition about
$g$ (cf. \cite{D3, D4, FN, Ca, Gu, INW}). At the same time, it generalizes the 
results from \cite{FK, He}.

We also note that verification of \ref{SC} does not require the existence of limits $g(\pm \infty)$
at all. As an example we consider $g(s) = \arctan s + c\cdot \cos s$ with an arbitrary constant $c \in \mathbb{R}$.
An easy calculation  yields that \eqref{11} has at least one solution for any $f \in L^{2}(\Omega)$ satisfying
\begin{equation}\label{14}
\left |\int\limits_{\Omega} f \phi \, {\rm d} x \right| < \frac{\pi}{2} \int\limits_{\Omega}| \phi | \, {\rm d} x
\end{equation}
for any $\phi \in \bar{H}$. On the other hand, the conditions \ref{LL} and various generalizations (see, e.g.
\cite{D1, D2}) do not apply in this case if $|c| \geq \frac{\pi}{2}$.

Above mentioned case $g(s) = \arctan s + c\cdot \cos s$ is covered by the so called potential
Landesman-Lazer condition:
\begin{equation}\label{PLL}\tag*{$(\rm PLL)_{\pm}$}
G^{\mp}\int\limits_{\Omega} \phi^{+} \, {\rm d} x - G^{\pm}\int\limits_{\Omega} \phi^{-} \, {\rm d} x <
\int\limits_{\Omega} \bar{f}\phi  \, {\rm d} x <
G^{\pm}\int\limits_{\Omega} \phi^{+} \, {\rm d} x - G^{\mp}\int\limits_{\Omega} \phi^{-} \, {\rm d} x
\end{equation}
where $G^{\pm} := \lim\limits_{s \rightarrow \pm \infty} \frac{G(s)}{s}$. Indeed, 
l'Hospital's rule implies $G^{-} = - \frac{\pi}{2}$, $G^{+} = \frac{\pi}{2}$ and the condition $(\rm PLL)_{+}$
reduces to \eqref{14}. For the use of \ref{PLL} see, e.g. the papers
\cite{BBR, T, T1, T2, T3, CLT}.

The conditions \ref{PLL} eliminate the influence of the bounded \textit{oscillating} term $c\cdot \cos s$
which disappears "in an average" as $|s| \rightarrow \infty$.

However, the conditions \ref{PLL} do not cover the case $g(s) = \frac{s}{1+ s^{2}} + c\cdot \cos s$,
where $c \in \mathbb{R}$ is an arbitrary constant. Indeed, both conditions
are empty, due to the fact $G^{\pm} = 0$. On the other hand, 
it follows from Theorem \ref{T11} that \eqref{11}
with $g$ given above
has a solution for any $f \in L^{2}(\Omega)^{\bot}$. This fact illustrates that our conditions
\ref{SC} refine also the conditions \ref{PLL} and, at the same time, they complement the results from \cite{NR1} and \cite{NR2}.

\begin{Ex}
The boundary value problem 
\begin{equation}\label{E}
\begin{aligned}
- \Delta u - \lambda_{k} u + \frac{u}{(e + u^{2}) \ln(e + u^{2})^{1/2}} 
+ c\cdot \cos u& = f \quad  {\text {in}} \; \Omega ,  \\
u & = 0 \quad  {\rm on} \; \partial\Omega,
\end{aligned}
\end{equation}
has a solution for arbitrary $c \in \mathbb{R}$ and for any $f \in L^{2}(\Omega)$ satisfying 
$$\int\limits_{\Omega} f \phi \, {\rm d} x = 0 $$
for any $\phi \in \bar{H}$. Indeed, since
$$\lim\limits_{|s| \rightarrow \infty} G(s) = \lim\limits_{|s| \rightarrow \infty} \left[ \ln \left( \ln (e+ s^{2})^{1/2}\right) + c \cdot \sin s \right] = \infty,$$
the result follows from Theorem \ref{T11}. On the other hand, the existence result for problems of type \eqref{E} does not follow from any work published in the literature so far.
\end{Ex}

%%%%%%%%%%%%%%%%%%%%%%%%%%%%%%%%%%%%%%%%%%%%%%%%%%%%%%%%%%%%%%%%%%%%%%%%%%%%%%%%%%%%%%%%%%%%%%%%%%%%%%%%%%%%%%%%%%
%%%%%%%%%%%%%%%%%%%%%%%%%%%%%%%%%%%%%%%%%%%%%%%%%%%%%%%%%%%%%%%%%%%%%%%%%%%%%%%%%%%%%%%%%%%%%%%%%%%%%%%%%%%%%%%%%% Section2
\section{Preliminaries}

In this section we stress some helpfull facts used in the proof of Theorem \ref{T11}.
\begin{Lm} \label{L21}
There exist $c_{1} >0$, $c_{2}>0$ such that for any $u \in H$ we have
\begin{equation}\label{21}
\int\limits_{\Omega} |\nabla \hat{u}|^{2} \, {\rm d}x- \lambda_{k}
\int\limits_{\Omega} (\hat{u} )^{2} \, {\rm d} x \leq - c_{1} \| \hat{u}\|^{2}
\end{equation}
and
\begin{equation}\label{22}
\left |\int\limits_{\Omega} g(u) \hat{u} \, {\rm d}x-
\int\limits_{\Omega} f \hat{u} \, {\rm d} x \right| \leq c_{2} \| \hat{u}\|.
\end{equation}
\end{Lm}
\begin{proof}
The inequality \eqref{21} follows from the variational characterization of $\lambda_{k}$, \eqref{22} follows from the H\H{o}lder
inequality, the boundedness of $g$ and the fact $f \in L^{2}(\Omega)$.
\end{proof}
\begin{Lm} \label{L22}
There exist $c_{3} >0$, $c_{4}>0$ such that for any $u \in H$ we have
\begin{equation}\label{23}
\int\limits_{\Omega} |\nabla \tilde{u}|^{2} \, {\rm d}x- \lambda_{k}
\int\limits_{\Omega} (\tilde{u} )^{2} \, {\rm d} x \geq c_{3} \| \tilde{u}\|^{2}
\end{equation}
and
\begin{equation}\label{24}
\left |\int\limits_{\Omega} g(u) \tilde{u} \, {\rm d}x-
\int\limits_{\Omega} f \tilde{u} \, {\rm d} x \right| \leq c_{4} \| \tilde{u}\|.
\end{equation}
\end{Lm}
\begin{proof}
The inequality \eqref{23} is also a consequence of the variational characterization of $\lambda_{k}$, and \eqref{24} follows
similarly as \eqref{22}.
\end{proof}
\begin{Lm} \label{L23}
There exist $c_{5} >0$ such that for any $u \in H$ we have
\begin{equation}\label{25}
\left |\int\limits_{\Omega} G(u) \, {\rm d}x-
\int\limits_{\Omega} f u \, {\rm d} x \right| \leq c_{5} \| u\|_{2}.
\end{equation}
\end{Lm}
\begin{proof}
The inequality \eqref{25} follows from the H\H{o}lder
inequality, the boundedness of $g$ and the fact $f \in L^{2}(\Omega)$.
\end{proof}

%%%%%%%%%%%%%%%%%%%%%%%%%%%%%%%%%%%%%%%%%%%%%%%%%%%%%%%%%%%%%%%%%%%%%%%%%%%%%%%%%%%%%%%%%%%%%%%%%%%%%%%%%%%%%%%%%%
%%%%%%%%%%%%%%%%%%%%%%%%%%%%%%%%%%%%%%%%%%%%%%%%%%%%%%%%%%%%%%%%%%%%%%%%%%%%%%%%%%%%%%%%%%%%%%%%%%%%%%%%%%%%%%%%%% Section3
\section{Proof of Theorem 1}

We define the \textit{energy functional} associated with \eqref{11}, $\mathcal{E}: H \rightarrow \mathbb{R}$, by
$$
\mathcal{E} (u) := \frac{1}{2} \int\limits_{\Omega} |\nabla u|^{2} \, {\rm d}x- \frac{\lambda_{k}}{2}
\int\limits_{\Omega} (u)^{2} \, {\rm d} x + \int\limits_{\Omega} G(u) \, {\rm d} x -
\int\limits_{\Omega} f u \, {\rm d} x,
$$
$u \in H$. Obviously, all critical points of $\mathcal{E}$ satisfy \eqref{13} and vice versa.

We will apply Saddle Point Theorem due to P. Rabinowitz \cite{R}:
\begin{Th} \label{T31}
Let $\mathcal{E} \in C^{1}(H, \mathbb{R})$ and $H = H^{-} \oplus H^{+}$, $\dim H^{-} < \infty$,
$\dim H^{+} = \infty$. Assume that
\begin{itemize}
  \item[\rm (a)] There exist a bounded neighborhood $D$ of $0$ in $H^{-}$ and a constant $\alpha \in \mathbb{R}$
  such that $\mathcal{E}\Bigr|_{\partial D} \leq \alpha$.
  \item[\rm (b)] There exists a constant $\beta > \alpha$ such that $\mathcal{E}\Bigr|_{H^{+}} \geq \beta$.
  \item[\rm (c)] $\mathcal{E}$ satisfies $(PS)$ condition.
\end{itemize}
Then the functional $\mathcal{E}$ has a critical point in $H$.
\end{Th}

At first we verify the Palais-Smale condition.
\begin{Lm} \label{L32}
Let us assume \ref{SC}. Then $\mathcal{E}$ satisfies $(PS)$ condition, i.e., if $(\mathcal{E}(u_{n})) \subset H$
is a bounded sequence and $\nabla \mathcal{E}(u_{n}) \rightarrow v$ in $H$, then there exist a subsequence
$(u_{n_{k}}) \subset (u_{n})$ and an element $u \in H$ such that $u_{n_{k}} \rightarrow u$ in $H$.
\end{Lm}
\begin{proof}
In the first step we prove that $(u_{n})$ is bounded in $L^{2}(\Omega)$. Assume the contrary, \textit{i.e.}, $\| u_{n}\|_{2} \rightarrow \infty$. Set $v_{n} := \frac{u_{n}}{\| u_{n} \|_{2}}$. Then
\begin{equation}\label{31}
\frac{\mathcal{E} (u_{n})}{\| u_{n} \|_{2}^{2}} := \frac{1}{2} \int\limits_{\Omega} |\nabla v_{n}|^{2} \, {\rm d}x- \frac{\lambda_{k}}{2}
\int\limits_{\Omega} {(v_{n})}^{2} \, {\rm d} x + \int\limits_{\Omega} \frac{G(u_{n})}{\| u_{n} \|_{2}^{2}} \, {\rm d} x -
\frac{1}{\| u_{n} \|_{2}} \int\limits_{\Omega} f v_{n} \, {\rm d} x \rightarrow 0.
\end{equation}
The second term is equal to $- \frac{\lambda_{k}}{2}$ since $\|v_{n} \|_{2} =1$, the last two terms go to zero
since
$$
\left| \frac{1}{\| u_{n} \|_{2}}\int\limits_{\Omega} f v_{n} \, {\rm d} x \right| \leq \frac{\| f \|_{2}}{\| u_{n} \|_{2}} \rightarrow 0
$$
and
\begin{equation*}%\label{1}
\begin{aligned}
\left| \int\limits_{\Omega} \frac{G(u_{n})}{\| u_{n} \|_{2}^{2}} \, {\rm d} x \right| & =
\frac{1}{\| u_{n} \|_{2}^{2}} \left| \int\limits_{\Omega} \left( \int\limits_{0}^{u_{n}(x)} g(s) \, {\rm d} s \right) \, {\rm d} x \right| \\
& \leq \frac{1}{\| u_{n} \|_{2}^{2}} \sup_{s\in \mathbb{R}} | g(s)| \cdot \int\limits_{\Omega} | u_{n}(x)| \, {\rm d} x
\leq \frac{\rm c}{\| u_{n} \|_{2}} \rightarrow 0
\end{aligned}
\end{equation*} (for some $c > 0$)
by the embedding $L^{2}(\Omega) \hookrightarrow  L^{1}(\Omega)$. Then it follows from \eqref{31}
that $(v_{n})$ is a bounded sequence in $H$. Passing to a subsequence, if necessary, we may assume that
there exists $v \in H$ such that $v_{n} \rightharpoonup v$ (weakly) in $H$ and $v_{n} \rightarrow v$ in $L^{2}(\Omega)$.

For arbitrary $w \in H$,
\begin{equation}\label{32}
\begin{aligned}
0 \leftarrow \frac{( \nabla \mathcal{E}'(u_{n}), w)}{\| u_{n} \|_{2}} = \int\limits_{\Omega} \nabla v_{n} \nabla w  \, {\rm d}x & - \lambda_{k}
\int\limits_{\Omega} v_{n}w \, {\rm d} x \\
& + \frac{1}{\| u_{n} \|_{2}} \int\limits_{\Omega} g(u_{n})w \, {\rm d} x -
\frac{1}{\| u_{n} \|_{2}} \int\limits_{\Omega} f w \, {\rm d} x.
\end{aligned}
\end{equation}
We have $\int\limits_{\Omega} \nabla v_{n} \nabla w  \, {\rm d}x  \rightarrow \int\limits_{\Omega} \nabla v \nabla w  \, {\rm d}x$ by $v_{n} \rightharpoonup v$ in $H$, $\int\limits_{\Omega} v_{n}w  \, {\rm d}x
\rightarrow \int\limits_{\Omega} v w  \, {\rm d}x$ by $v_{n} \rightarrow v$ in $L^{2}(\Omega)$,
$\frac{1}{\| u_{n} \|_{2}} \int\limits_{\Omega} f w \, {\rm d} x \rightarrow 0$, $ \frac{1}{\| u_{n} \|_{2}} \int\limits_{\Omega} g(u_{n})w \, {\rm d} x \rightarrow 0$ by $f \in L^{2}(\Omega)$, the boundedness of $g$ and by our assumption
$\| u_{n}\|_{2} \rightarrow \infty$. Then it follows from \eqref{32} that
$$
\int\limits_{\Omega} \nabla v \nabla w  \, {\rm d}x - \lambda_{k}
\int\limits_{\Omega} v w \, {\rm d} x =0
$$
holds for arbitrary $w \in H$, \textit{i.e.}, $v= \phi_{0} \in \bar{H}$ is an eigenfunction associated with $\lambda_{k}$.
That is, $\frac{u_{n}}{\| u_{n} \|_{2}} \rightarrow \phi_{0}$ in $L^{2}(\Omega)$.

Now, by the assumption $\nabla \mathcal{E}(u_{n}) \rightarrow o$ and the orthogonal decomposition of $H$,
\begin{equation}\label{33}
\begin{aligned}
o(\| \hat{u}_{n} \|) = ( \nabla \mathcal{E}(u_{n}), \hat{u}_{n}) = \int\limits_{\Omega} | \nabla \hat{u}_{n}|^{2} \, {\rm d}x & - \lambda_{k} \int\limits_{\Omega} (\hat{u}_{n})^{2} \, {\rm d} x \\
& + \int\limits_{\Omega} g(u_{n})\hat{u}_{n} \, {\rm d} x - \int\limits_{\Omega} f \hat{u}_{n} \, {\rm d} x.
\end{aligned}
\end{equation}
By Lemma \ref{L21} it follows from \eqref{33} that
$$o(1) \leq - c_{1} \| \hat{u}_{n} \| + c_{2}$$
with $c_{1}, c_{2} > 0$ independent of $n$. Hence $\| \hat{u}_{n} \| $ is a bounded sequence.

Similarly, we also have
\begin{equation}\label{34}
\begin{aligned}
o(\| \tilde{u}_{n} \|) = ( \nabla \mathcal{E}(u_{n}), \tilde{u}_{n}) = \int\limits_{\Omega} | \nabla \tilde{u}_{n}|^{2} \, {\rm d}x & - \lambda_{k} \int\limits_{\Omega} (\tilde{u}_{n})^{2} \, {\rm d} x \\
& + \int\limits_{\Omega} g(u_{n})\tilde{u}_{n} \, {\rm d} x - \int\limits_{\Omega} f \tilde{u}_{n} \, {\rm d} x.
\end{aligned}
\end{equation}
By Lemma \ref{L22} it follows from \eqref{34} that
$$o(1) \geq c_{3} \| \tilde{u}_{n} \| - c_{4}$$
with $c_{3}, c_{4} > 0$ independent of $n$. Hence $\| \tilde{u}_{n} \|$ is a bounded sequence. Let us split now $\mathcal{E}(u_{n})$ as follows
\begin{equation*}
\begin{aligned}
\mathcal{E} (u_{n}) & = \underbrace{\frac{1}{2} \int\limits_{\Omega} |\nabla \hat{u}_{n}|^{2} \, {\rm d}x- \frac{\lambda_{k}}{2}
\int\limits_{\Omega} (\hat{u}_{n})^{2} \, {\rm d} x}_{A} + \underbrace{\frac{1}{2} \int\limits_{\Omega} |\nabla \tilde{u}_{n}|^{2} \, {\rm d}x  - \frac{\lambda_{k}}{2}
\int\limits_{\Omega} (\tilde{u}_{n})^{2} \, {\rm d} x}_{B} \\
& + \underbrace{\int\limits_{\Omega} G(u_{n})\, {\rm d} x -
 \int\limits_{\Omega} \bar{f} u_{n} \, {\rm d} x}_{C} - \underbrace{\int\limits_{\Omega} f^{\bot} \hat{u}_{n} \, {\rm d} x
 -\int\limits_{\Omega} f^{\bot} \tilde{u}_{n} \, {\rm d} x}_{D}.
\end{aligned}
\end{equation*}
The boundedness of $\| \hat{u}_{n} \|$ and $\| \tilde{u}_{n} \|$ implies that $A, B$ and $D$ are bounded
terms. On the other hand, $(\rm SC)_{+}$ forces $C \rightarrow + \infty$ and $(\rm SC)_{-}$ forces $C \rightarrow -\infty$. In particular, we conclude
$\mathcal{E} (u_{n}) \rightarrow \pm \infty$ which contradicts the assumption of the boundedness of
$(\mathcal{E} (u_{n}))$. We thus proved that $(u_{n})$ is a bounded sequence in $L^{2}(\Omega)$.

In the second step we select a strongly convergent subsequence (in $H$) from $(u_{n})$. Let us examine again the terms in
$$
\mathcal{E} (u_{n}) := \frac{1}{2} \int\limits_{\Omega} |\nabla u_{n}|^{2} \, {\rm d}x- \frac{\lambda_{k}}{2}
\int\limits_{\Omega} (u_{n})^{2} \, {\rm d} x + \int\limits_{\Omega} G(u_{n}) \, {\rm d} x -
\int\limits_{\Omega} f u_{n} \, {\rm d} x.
$$
By the assumption $\mathcal{E} (u_{n}) $ is bounded. The boundedness of the sequence $(u_{n})$ in $L^{2}(\Omega)$ implies that
$\int\limits_{\Omega} (u_{n})^{2} \, {\rm d} x $, $ \int\limits_{\Omega} G(u_{n}) \, {\rm d} x$ and
$ \int\limits_{\Omega} f u_{n} \, {\rm d} x$ are bounded independently of $n$, as well. Therefore,
$ \| u_{n}\|^{2} = \int\limits_{\Omega} |\nabla u_{n}|^{2} \, {\rm d}x$ must be also bounded. Hence,
we may assume, without lost of generality, that $u_{n} \rightharpoonup u$ in $H$ for some $u \in H$, and
$u_{n} \rightarrow u$ in $L^{2}(\Omega)$. Then
\begin{equation*}%\label{32}
\begin{aligned}
0 \leftarrow ( \nabla \mathcal{E}(u_{n}), u_{n} - u)= \int\limits_{\Omega} \nabla u_{n} \nabla (u_{n} - u) \, {\rm d}x & - \lambda_{k}
\int\limits_{\Omega} u_{n}(u_{n} - u)\, {\rm d} x \\
& + \int\limits_{\Omega} g(u_{n})(u_{n} - u) \, {\rm d} x -
\int\limits_{\Omega} f (u_{n} - u) \, {\rm d} x.
\end{aligned}
\end{equation*}
Since
$$
- \lambda_{k}
\int\limits_{\Omega} u_{n}(u_{n} - u)\, {\rm d} x
+ \int\limits_{\Omega} g(u_{n})(u_{n} - u) \, {\rm d} x -
\int\limits_{\Omega} f (u_{n} - u) \, {\rm d} x \rightarrow 0,
$$
we conclude that
$$
\int\limits_{\Omega} \nabla u_{n} \nabla (u_{n} - u) \, {\rm d}x \rightarrow 0
$$
as well. So,
$$
\int\limits_{\Omega} |\nabla u_{n}|^{2} \, {\rm d}x - \int\limits_{\Omega} \nabla u_{n} \nabla u \, {\rm d}x  \rightarrow 0
$$
which together with
$$
\int\limits_{\Omega} \nabla u_{n} \nabla u \, {\rm d}x  \rightarrow \| u_{n}\|^{2}
$$
(this is due to the weak convergence $u_{n} \rightharpoonup u$) yields
$$
\| u_{n}\| \rightarrow \| u\|.
$$
The uniform convexity of $H$ then implies that $u_{n} \rightarrow u$ in $H$. Hence $\mathcal{E}$ satisfies the condition
$({\rm c})$ in Theorem \ref{T31}.
\end{proof}

Now we prove that also $({\rm a})$ and $({\rm b})$ hold. To this end we have to consider separately the case
$(\rm SC)_{+}$ and $(\rm SC)_{-}$.

\vspace{5mm} \noindent
1. Let us assume that $(\rm SC)_{+}$ holds. We set
$$ H^{-} := \hat{H}, \quad H^{+} := \bar{H} \oplus \tilde{H}. $$
It follows from Lemma \ref{L21} and \ref{L23} that
\begin{equation}\label{35}
\begin{aligned}
\lim\limits_{\| \hat{u}\| \rightarrow \infty}
\mathcal{E} (\hat{u}) := \lim\limits_{\| \hat{u}\| \rightarrow \infty} \left[ \frac{1}{2} \int\limits_{\Omega} |\nabla \hat{u}|^{2} \, {\rm d}x \right. & - \frac{\lambda_{k}}{2}
\int\limits_{\Omega} (\hat{u})^{2} \, {\rm d} x \\
& \left. + \int\limits_{\Omega} G(\hat{u}) \, {\rm d} x -
\int\limits_{\Omega} f \hat{u} \, {\rm d} x\right] = -\infty.
\end{aligned}
\end{equation}
On the other hand, we prove that there exists $\beta \in \mathbb{R}$ such that
$$
\inf\limits_{u \in H^{+}} \mathcal{E} (u) \geq \beta.
$$
Assume the contrary, that is, there exists a sequence $(u_{n}) \subset H^{+}$ such that
\begin{equation}\label{36}
\begin{aligned}
\lim\limits_{n \rightarrow \infty}
\mathcal{E} (u_{n}) = - \infty.
\end{aligned}
\end{equation}
Then  $\| u_{n}\|_{2} \rightarrow \infty$, and for $v_{n} := \frac{u_{n}}{\| u_{n} \|_{2}}$
($v_{n} \in H^{+}$) we have
\begin{equation}\label{37}
\begin{aligned}
0 \geq \limsup\limits_{n \rightarrow \infty}
\frac{\mathcal{E} (u_{n})}{\| u_{n}\|_{2}^{2}} := \limsup\limits_{n \rightarrow \infty} \left[ \frac{1}{2} \int\limits_{\Omega} |\nabla v_{n}|^{2} \, {\rm d}x \right. & - \frac{\lambda_{k}}{2}
\int\limits_{\Omega} (v_{n})^{2} \, {\rm d} x \\
& \left. + \int\limits_{\Omega} \frac{G(u_{n})}{\| u_{n}\|_{2}^{2}} \, {\rm d} x -
\int\limits_{\Omega} f \frac{v_{n}}{\| u_{n}\|_{2}} \, {\rm d} x\right].
\end{aligned}
\end{equation}
Clearly, by Lemma \ref{L23}, we have
\begin{equation}\label{38}
\begin{aligned}
\int\limits_{\Omega} \frac{G(u_{n})}{\| u_{n}\|_{2}^{2}} \, {\rm d} x -
\int\limits_{\Omega} f \frac{v_{n}}{\| u_{n}\|_{2}} \, {\rm d} x \rightarrow 0 .
\end{aligned}
\end{equation}
It follows from \eqref{37} and \eqref{38} that $\| v_{n}\|$ is bounded. Passing to
a subsequence if necessary, we may assume that there exists $v \in H^{+}$ such that
$v_{n} \rightharpoonup v$ in $H$ and $v_{n} \rightarrow v$ in $L^{2}(\Omega)$. Moreover,
\begin{equation}\label{39}
\begin{aligned}
\liminf\limits_{n \rightarrow \infty}
\int\limits_{\Omega} |\nabla v_{n}|^{2} \, {\rm d}x \geq \int\limits_{\Omega} |\nabla v|^{2} \, {\rm d}x
\end{aligned}
\end{equation}
by the weak lower semicontinuity of the norm in $H$.
We deduce from \eqref{37} - \eqref{39} that
$$
\int\limits_{\Omega} |\nabla v|^{2} \, {\rm d}x  - \lambda_{k}
\int\limits_{\Omega} (v)^{2} \, {\rm d} x \leq 0,
$$
and hence, from Lemma \ref{L22}, it follows that $v = \phi_{0} \in \bar{H}$ is an eigenfunction
associated with $\lambda_{k}$. That is,
$$
\frac{u_{n}}{\| u_{n}\|_{2}} \rightarrow \phi_{0} \quad {\text {in}} \quad L^{2}(\Omega).
$$
By Lemma \ref{L22}, by the properties of the orthogonal decomposition of $H^{+}$ and $f$ and
by the condition $(\rm SC)_{+}$, we have for $u_{n} \in H^{+}$:
$$
\begin{aligned}
\lim\limits_{n \rightarrow \infty}
\mathcal{E} (u_{n}) := & \lim\limits_{n \rightarrow \infty} \left[ \frac{1}{2} \int\limits_{\Omega} |\nabla u_{n}|^{2} \, {\rm d}x \right. - \frac{\lambda_{k}}{2}
\int\limits_{\Omega} (u_{n})^{2} \, {\rm d} x 
 \left. + \int\limits_{\Omega} G(u_{n}) \, {\rm d} x -
\int\limits_{\Omega} f u_{n} \, {\rm d} x\right] \\
= & \lim\limits_{n \rightarrow \infty} \left[ \frac{1}{2} \int\limits_{\Omega} |\nabla \tilde{u}_{n}|^{2} \, {\rm d}x \right. - \frac{\lambda_{k}}{2}
\int\limits_{\Omega} (\tilde{u}_{n})^{2} \, {\rm d} x \\
& \left. + \int\limits_{\Omega} G(u_{n}) \, {\rm d} x -
\int\limits_{\Omega} \bar{f} u_{n} \, {\rm d} x - \int\limits_{\Omega} f^{\bot} \tilde{u}_{n} \, {\rm d} x\right] \\
 \geq  & \lim\limits_{n \rightarrow \infty} \left[c_{3} \| \tilde{u}_{n}\|^{2} - \| f^{\bot}\|_{2} 
\| \tilde{u}_{n}\|_{2} \right] + \lim\limits_{n \rightarrow \infty} \left[ \int\limits_{\Omega} G(u_{n}) \, {\rm d} x -
\int\limits_{\Omega} \bar{f} u_{n} \, {\rm d} x \right] \\
 = & + \infty.
\end{aligned}
$$
This contradicts \eqref{36}.

By \eqref{35} there exists $R >0$ such that for $ D := \{ u \in H^{-}: \| u \|\leq R \} $ the following inequality 
holds
$$
\sup\limits_{u \in \, \partial D} \mathcal{E} (u) < \alpha := \beta -1.
$$
Hence, we proved (a) and (b) in Theorem \ref{T31}.

\vspace{5mm} \noindent
2. Let us assume that $(\rm SC)_{-}$ holds. In this case we set
$$ H^{-} := \hat{H} \oplus \bar{H}, \quad H^{+} := \tilde{H}. $$
Let $u \in H^{+}$. Then by Lemmas \ref{L22} and \ref{L23} we have
\begin{equation*}%\label{35}
\begin{aligned}
\mathcal{E} (u) & :=  \frac{1}{2} \int\limits_{\Omega} |\nabla u|^{2} \, {\rm d}x - \frac{\lambda_{k}}{2}
\int\limits_{\Omega} (u)^{2} \, {\rm d} x + \int\limits_{\Omega} G(u) \, {\rm d} x -
\int\limits_{\Omega} f u \, {\rm d} x\\ 
& \; \geq c_{3} \| u\|^{2} - c_{5} \| u \|_{2} \geq c_{3} \| u\|^{2} - c_{6} \| u \|.
\end{aligned}
\end{equation*}
Hence there exists $\beta \in \mathbb{R}$ such that $\mathcal{E} (u)\geq \beta$ for all $u \in H^{+}$.
On the other hand, we prove that 
\begin{equation}\label{310}
\lim\limits_{\| u \| \rightarrow \infty, u \in H^{-}} \mathcal{E} (u) = - \infty.
\end{equation}
Notice, that $\dim H^{-} < \infty$ implies that the norms $\| \cdot \| $ and $\| \cdot \|_{2}$ are equivalent
on $H^{-}$. Assume by the contradiction that \eqref{310} does not hold, \textit{i.e.}, 
there exist a sequence $(u_{n}) \subset H^{-}$ and a constant $c \in \mathbb{R}$ such that
$\| u_{n}\|_{2} \rightarrow \infty$ and
\begin{equation}\label{311}
\mathcal{E} (u_{n}) \geq c.
\end{equation}
Set $v_{n} := \frac{u_{n}}{\| u_{n} \|_{2}}$. Due to $\dim H^{-} < \infty$ we may assume that there
exists $v \in H^{-}$ such that $v_{n} \rightarrow v$ both in $H$ and $L^{2}(\Omega)$. Then 
\begin{equation}\label{312}
\begin{aligned}
0 \leq & \liminf\limits_{n \rightarrow \infty}
\frac{\mathcal{E} (u_{n})}{\| u_{n}\|_{2}^{2}} = \liminf\limits_{n \rightarrow \infty} \left[ \frac{1}{2} \int\limits_{\Omega} |\nabla v_{n}|^{2} \, {\rm d}x \right. - \frac{\lambda_{k}}{2}
\int\limits_{\Omega} (v_{n})^{2} \, {\rm d} x \\
& \left. + \int\limits_{\Omega} \frac{G(u_{n})}{\| u_{n}\|_{2}^{2}} \, {\rm d} x -
\int\limits_{\Omega} f \frac{v_{n}}{\| u_{n}\|_{2}} \, {\rm d} x\right] 
= \frac{1}{2} \int\limits_{\Omega} |\nabla v|^{2} \, {\rm d}x  - \frac{\lambda_{k}}{2}
\int\limits_{\Omega} (v)^{2} \, {\rm d} x,
\end{aligned}
\end{equation}
by Lemma \ref{L23}. According to Lemma \ref{L21}, \eqref{312} implies $v = \phi_{0} \in \bar{H}$,
an eigenfunction associated with $\lambda_{k}$. Hence $\frac{u_{n}}{\| u_{n} \|_{2}} \rightarrow \phi_{0}$
in $L^{2}(\Omega)$. Now, it follows from the orthogonal decomposition of $H^{-}$ and $f$, Lemma \ref{L21}
and $(\rm SC)_{-}$ that for $u_{n} \in H^{-}$,
$$
\begin{aligned}
\lim\limits_{n \rightarrow \infty}
\mathcal{E} (u_{n}) := & \lim\limits_{n \rightarrow \infty} \left[ \frac{1}{2} \int\limits_{\Omega} |\nabla u_{n}|^{2} \, {\rm d}x \right. - \frac{\lambda_{k}}{2}
\int\limits_{\Omega} (u_{n})^{2} \, {\rm d} x
 \left. + \int\limits_{\Omega} G(u_{n}) \, {\rm d} x -
\int\limits_{\Omega} f u_{n} \, {\rm d} x\right] \\
 = & \lim\limits_{n \rightarrow \infty} \left[ \frac{1}{2} \int\limits_{\Omega} |\nabla \hat{u}_{n}|^{2} \, {\rm d}x \right. - \frac{\lambda_{k}}{2}
\int\limits_{\Omega} (\hat{u}_{n})^{2} \, {\rm d} x \\
& \left. + \int\limits_{\Omega} G(u_{n}) \, {\rm d} x -
\int\limits_{\Omega} \bar{f} u_{n} \, {\rm d} x - \int\limits_{\Omega} f^{\bot} \hat{u}_{n} \, {\rm d} x\right] \\
\leq & \lim\limits_{n \rightarrow \infty} \left[- c_{1} \| \hat{u}_{n}\|^{2} + c_{2}
\| \hat{u}_{n}\| \right] + \lim\limits_{n \rightarrow \infty} \left[ \int\limits_{\Omega} G(u_{n}) \, {\rm d} x -
\int\limits_{\Omega} \bar{f} u_{n} \, {\rm d} x \right] \\
= & - \infty.
\end{aligned}
$$
This contradicts \eqref{311}, \textit{i.e.}, \eqref{310} holds true. Let us choose again
$ D := \{ u \in H^{-}: \| u \|\leq R \} $. Then, for $R >0$ large enough, we have
$$
\sup\limits_{u \in \, \partial D} \mathcal{E} (u) < \alpha := \beta -1.
$$
and (a) and (b) in Theorem \ref{T31} are proved.

Recall that the hypothesis (c) in Theorem \ref{T31} is proved in Lemma \ref{L32}
for both cases \ref{SC}. It then follows from Theorem \ref{T31} that under
the assumptions \ref{SC} there exists a critical point of $\mathcal{E}$. Since this is also a solution
of \eqref{11}, the proof of Theorem \ref{T11} is finished.

%%%%%%%%%%%%%%%%%%%%%%%%%%%%%%%%%%%%%%%%%%%%%%%%%%%%%%%%%%%%%%%%%%%%%%%%%%%%%%%%%%%%%%%%%%%%%%%%%%%%%%%%%%%%%%%%%%%%%%%¡%%%%%%%%%%%%%%%%%
%%%%%%%%%%%%%%%%%%%%%%%%%%%%%%%%%%%%%%%%%%%%%%%%%%%%%%%%%%%%%%%%%%%%%%%%%%%%%%%%%%%%%%%%%%%

\vspace{10mm}

\noindent {\bf Acknowledgment.}
This research was
supported by the Grant 13-00863S of the Grant Agency
of Czech Republic and by the European Regional Development Fund (ERDF), project "NTIS – New Technologies for the Information Society", European Centre of Excellence, CZ.1.05/1.1.00/02.0090.


\begin{thebibliography}{15}

\bibitem{DR}
P. Dr\'abek, S.B. Robinson, On the solvability of resonance problems with respect to the Fu\v{c}\'{i}k Spectrum, J. Math. Anal. Appl. 418 (2014), 884-905.

\bibitem{LL}
E.M. Landesman, A.C. Lazer, Nonlinear perturbations of linear elliptic boundary value problems at resonance, J. Math. Mech. 19 (1970), 609-623.

\bibitem{D}
P. Dr\'abek, Solvability and Bifurcations of Nonlinear Equations,
Longman Scientific \& Technical, Pitman Res. Notes in Math. Series
264, Longman, 1992.

\bibitem{D3}
P. Dr\'abek, Existence and multiplicity results for some weakly nonlinear elliptic problems
at resonance, \v{C}as. p\v{e}st. mat. (Math. Bohemica) 108 (1983), 272-284.

\bibitem{D4}
P. Dr\'abek, Bounded nonlinear perturbations of second order linear elliptic problems,
Comment. Math. Univ. Carolinse 22, 2 (1981), 215-221.

\bibitem{FN}
D.G. de Figueiredo, W.M. Ni, Perturbations of second order linear elliptic problems by nonlinearities without Landesman-Lazer condition, Nonlinear Anal. Theory, Methods and Applications  3 (1979), 629-634.

\bibitem{Ca}
A. Ca\~{n}ada, K-set contractions and nonlinear vector boundary value problems, Journal of Mathematical Analysis and Applications 117 (1986), 1-22.

\bibitem{Gu}
C.P. Gupta, Solvability of a boundary value problem with the nonlinearity satisfying a sign condition, Journal of Mathematical Analysis and Applications 129 (1988), 482-492.

\bibitem{INW}
R. Iannacci, M.N. Nkashama, J.R. Ward, Nonlinear second order elliptic partial differential equations at resonance, Trans. Am. math. Soc. 311 (1989), 711-726.

\bibitem{FK}
S. Fu\v{c}\'{\i}k, M. Krbec, Boundary value problems with bounded nonlinearity and general nullspace of linear part, Mathematische Zeitschrift 155 (1977), 129-138.

\bibitem{He}
P. Hess, A remark on the preceding paper of Fu\v{c}\'{\i}k and Krbec, Mathematische Zeitschrift 155 (1977), 139-141. 

\bibitem{D1}
P. Dr\'abek, On the resonance problem with nonlinearity which has arbitrary linear growth,
J. Math. Anal. Appl. 127 (1987), 435-442.

\bibitem{D2}
P. Dr\'abek, Landesman-Lazer condition for nonlinear problems with jumping nonlinearities,
J. Differential Equations 85 (1990), 186-199.

\bibitem{BBR}
A.A. Bliss, J. Buerger, A.J. Rumbos, Periodic boundary-value problems and Dancer-Fu\v{c}\'{i}k spectrum under conditions of resonance, Electron. J. Differential Equations 112 (2011), 1-34.

\bibitem{T}
P. Tomiczek, Potential Landesman-Lazer type conditions and the Fucik spectrum, Electron. J. Diff. Eqns. 94 (2005), 1-12.

\bibitem{T1}
P. Tomiczek, The Duffing equation with the potential Landesman-Lazer condition, 
Nonlinear Analysis: Theory, Methods and Applications 70 (2009), 735-740.

\bibitem{T2}
P. Tomiczek, Periodic Problem with a Potential Landesman Lazer Condition, Boundary Value Problems 2010, 2010:586971.  doi:10.1155/2010/586971

\bibitem{T3}
P. Tomiczek, A generalization of the Landesman-Lazer condition, Electron. J. Diff. Eqns. 04 (2001), 1-11.

\bibitem{CLT}
C.L. Tang, Solvability for Two-Point Boundary Value Problems, J. Math. Anal. Appl. 216 (1997), 368-374.

\bibitem{NR1} 
M.N. Nkashama, S.B. Robinson, Resonance and Nonresonance in Terms of Average Values, 
J. Differential Equations 132 (1996), 46-65.

\bibitem{NR2} 
M.N. Nkashama, S.B. Robinson, Resonance and nonresonance in terms of average values. II, 
Proc. Roy. Soc. Edinburgh Sect. A 131 (2001), no. 5, 1217-1235.

\bibitem{R}
P.H. Rabinowitz, Minimax Methods in Critical Point Theory with Applications
to Differential Equations, Amer. Math. Soc., Providence, RI, 1986.



\end{thebibliography}
\end{document}